\documentclass[11pt]{csdm}

\usepackage{url,floatflt,helvet,times}
\usepackage{psfig,graphics}
\usepackage{mathptmx,amssymb,amsmath,bm}
\usepackage{float,tcolorbox,xcolor}
\usepackage[bf,hypcap]{caption}
\usepackage{graphicx,wrapfig,subfig}
\usepackage{setspace,enumerate,bookmark,placeins}
\usepackage{multirow,multicol}
\usepackage[curve,frame,line,arrow,matrix]{xy}
\usepackage[latin5]{inputenc}
\usepackage{tikz,calc}
\usetikzlibrary{arrows,backgrounds,automata}
\usetikzlibrary{decorations.markings}
\tikzstyle{vertex}=[circle, draw, inner sep=2pt, minimum size=6pt]

\emergencystretch=\maxdimen
\allowdisplaybreaks

\def\noi{\noindent}

\allowdisplaybreaks

\def\firstpage{37}
\def\thevol{1}
\def\thenumber{1}
\def\theyear{2017}

\begin{document}
{\fontfamily{cmss}\selectfont

\titlefigurecaption{\hspace{0.3cm}{\hspace{0.6cm}\LARGE \bf \sc \sffamily\color{white} Contemporary Studies in Discrete Mathematics}}

\title{\sc \sffamily On the Fading Number of a Graph}

\author{\sc\sffamily Johan kok$^{1,\ast}$, Sudev Naduvath$^2$ and Eunice Gogo Mphako-Banda$^3$}
\institute {$^{1,2}$ Centre for Studies in Discrete Mathematics, Vidya Academy of Science \& Technology, Thalakkottukara,Thrissur, India.\\
$^{3}$School of Mathematical Sciences, University of Witswatersrand, Johannesburg, South Africa.
}


\titlerunning{On the Fading Number of a Graph}
\authorrunning{J. Kok, S.Naduvath \& E.G. Mphako-Banda}

\mail{kokkiek2@tshwane.gov.za}

\received{26 August 2017}
\revised{28 September 2017}
\accepted{05 October 2017}
\published{20 October 2017.}

\abstracttext{The closed neighbourhood $N[v]$ of a vertex $v$ of a graph $G$, consisting of at least one vertex from all colour classes with respect to a proper colouring of $G$, is called a rainbow neighbourhood in $G$.  The minimum number of vertices and the maximum number of vertices which yield rainbow neighbourhoods with respect to a chromatic colouring of $G$ are called the minimum and maximum rainbow neighbourhood numbers, denoted by $r^-_\chi(G)$, $r^+_\chi(G)$ respectively. In this paper, by a colour, we mean a solid colour and by a transparent colour, we mean the fading of a solid colour. The fading numbers of a graph $G$, denoted by $f^-(G)$, $f^+(G)$ respectively, are the maximum number of vertices for which the colour may fade to transparent without a decrease in $r^-_\chi(G)$ and $r^+_\chi(G)$ respectively.}

\maketitle

\section{Introduction}

For all terms and definitions, which are not defined specifically in this paper, we refer to \cite{BM1,FH,DBW}. Unless mentioned otherwise, all graphs considered here are undirected, simple, finite and connected.

The \textit{open neighbourhood} of a vertex $v \in V(G)$, denoted by $N(v)$, is the set of all vertices that are adjacent to $v$. The closed neighbourhood of $v$, denoted by $N[v]$, is defined as $N(v)\cup \{v\}$.

A \textit{vertex colouring} of a graph $G$ is an assignment $\varphi:V(G) \mapsto \mathcal{C}= \{c_1,c_2,\ldots,c_\ell\}$, of a set of colours to the vertices of $G$ and is denoted by, $\varphi(G)$. A vertex colouring $\varphi(G)$ is said to be a \textit{proper vertex colouring} of a graph $G$ if no two adjacent vertices have the same colour. The minimum number of colours in a proper colouring of $G$ is called the \textit{chromatic number} of $G$ and is denoted by $\chi(G)$. We call a colouring consisting of $\chi(G)$ colours a \textit{$\chi$-colouring} or a \textit{chromatic colouring} of $G$. When the context is clear the corresponding chromatic colouring is written as $c:V(G)\mapsto \mathcal{C}$ or as $c(V(G)) = \mathcal{C}$. Generally, when the cardinality of the set of colours $\mathcal{C}$ is bound by conditions such as minimum, maximum or others and since, $c(V(G)) = \mathcal{C}$, it can be agreed that $c(G)$ means $c(V(G))$ hence, $c(G) \Rightarrow \mathcal{C}$ and $|c(G)| = |\mathcal{C}|$. 

The set of all vertices of a graph $G$ having a particular colour $c_i$ with respect to a chromatic colouring $c(G)$ is called the colour class of $c_i$ and is denoted by, $\mathcal{C}_i$.

Unless mentioned otherwise, in the following discussion, we follow the following convention in colouring the vertices of a graph $G$. Let $I_1$ be the maximal independent in $G$. Assign colour $c_1$ to all vertices in $I_1$. Let $G_1=G-I_1$ and let $I_2$ be a maximal independent set in $G_1$.  Assign colour $c_2$ to all vertices in $I_2$. Proceed like this until all vertices in $G$ are coloured properly. This convention is called \textit{rainbow neighbourhood convention} (see \cite{KNJ}).

The closed neighbourhood $N[v]$ of a vertex $v \in V(G)$, consisting of at least one vertex from all colour classes, with respect to a chromatic colouring $c(G)$, is called a \textit{rainbow neighbourhood} in $G$. The number of vertices in $G$ which yield rainbow neighbourhoods in $G$ is called the \textit{rainbow neighbourhood number} of $G$ and is denoted by $r_\chi(G)$. For further studies on rainbow neighbourhood number of graphs, see \cite{KNJ,KNB,NSKK,SNKK}. Note that $r^-_\chi(G)$ and $r^+_\chi(G)$ respectively denote the minimum value and maximum value of $r_\chi(G)$ over all minimum proper colourings (see \cite{KNB}).

A \textit{minimum parameter colouring} of a graph $G$ is a proper colouring of $G$ which consists of the colours $c_i;\ 1\le i\le \chi(G)$ or put differently, $i = 1, 2, 3,\ldots, \chi(G)$. In other words, a minimum parameter colouring of $G$ is the colouring consisting of the colours having minimum subscripts. Unless stated otherwise, we consider minimum parameter colouring throughout this paper. 
Note that $r^-_\chi(G)$ necessarily corresponds to a chromatic colouring in accordance with the rainbow neighbourhood convention. 


\section{Fading Number of a Graph}

In this discussion, by a colour we mean a solid colour (or an opaque colour), while  \textit{transparent colour} is considered to be the fading of a solid colour. We shall denote a transparent colour by $c^\circ$ and note that $c^\circ \notin \mathcal{C}$. Therefore, for any vertex $v$ with $c(v)=c^\circ$ and adjacent to a vertex $u$, we have $c(v)\notin c(N[u])$. 

\noi The notion of the fading number of a graph $G$ is defined as:

\begin{definition}{\rm
The \textit{fading number} corresponding to a chromatic colouring of $G$ is the maximum number of vertices for which their solid colours may fade to transparent without any decrease in its rainbow neighbourhood number $r_\chi(G)$. 
}\end{definition}

A \textit{fade set} of graph $G$ corresponding to a chromatic colouring is defined as  $F^\circ(G) = \{v: c(v) = c^\circ\}$.

Let $F^\circ(G)$ be a fade set of graph $G$ corresponding to a chromatic colouring $c(G)$ defined by $F^\circ(G) = \{v: c(v) = c^\circ\}$. Also, let $f^-(G) =\max\{|F^\circ(G)|:r^-_\chi (G) = r^-_\chi(\langle V(G) - F^\circ(G)\rangle)\}$, where $r^-_\chi(G)$ is taken over all possible fade sets $F^\circ(G)$.  Similarly, let $f^+(G) =\max\{|F^\circ(G)|:r^+_\chi (G) = r^+_\chi(\langle V(G) - F^\circ(G)\rangle)\}$, where $r^+_\chi(G)$ is taken over all possible fade sets $F^\circ(G)\}$. 

Note that since both $f^-(G)$ and $f^+(G)$ are maximum values over similar sets which differ only because of the number of times colours in a chromatic colouring are allocated to vertices. Hence, we have $f^-(G) \geq f^+(G)$.

Consider the thorn cycle $C^\star_3$ which has cycle vertices $v_1,v_2,v_3$ with  $t_i\geq1$ pendant vertices at the vertex $v_i$, where $i=1,2,3$. Clearly, $\chi(C^\star_3)=3$ and $r^-_\chi(C^\star_3)=r^+_\chi(C^\star_3)=3$. The colours assigned to the $\sum\limits_{i=1}^{3}t_i$ pendant vertices may all fade to $c^\circ$ without a decrease in the respective rainbow neighbourhood numbers. This example establishes the existence of infinitely many graphs with arbitrary large fading numbers.

For $n\ge 1$, both the null graph $\mathfrak{N}_n$ and the complete graph $K_n$, are the extremal graphs with regards to the structor index $si(G)$ of a graph $G$ and hence we have $f^-(\mathfrak{N}_n)=f^+(\mathfrak{N}_n)=0$ and $f^-(K_n)=f^+(K_n)=0$.

Solid colours may represent a specific property abstraction, technology type or infrastructure type all of which are subject to deterioration. This deterioration is modelled to be fading of solid colours. The question to be answered is when critical deterioration has established in a graph or network such that fully functional subgraphs or subnetworks are retained. 

From the definition and concepts mentioned above, we have the following theorem, which sets a lower bound to the fading number of a general chromatic colouring of a graph $G$.

\begin{theorem}\label{Thm-2.1}
Let $c(G)$ be a general chromatic colouring of a graph $G$ of order $n$ and let $U$ be the set of all vertices $u$ in $G$ such that $N[u]$ is rainbow neighbourhoods in $G$. Then, the general fading number is of $G$ is $f_\chi(G) \geq n-|\bigcup\limits_{u\in U} N[u]|$.
\end{theorem}
\begin{proof}~
Note that $|\bigcup\limits_{u\in U}N[u]|$ equals the exact number of vertices which belongs to some rainbow neighbourhood in $G$. Hence, we have $n \geq |\bigcup\limits_{u\in U}N[u]|$ and $n-|\bigcup\limits_{u\in U}N[u]|  \geq 0$. Since, for some $w\in U$, $c(N[w])$ may contain repetition of a colour, a corresponding vertex (or vertices) may fade to transparent. Therefore, $f_\chi(G)\ge n-|\bigcup\limits_{u\in U}N[u]|$.
\end{proof}

Note that the equality holds in Theorem \ref{Thm-2.1} when $N[u_i]\cap N[u_j]=\emptyset$ for all distinct pairs $u_i,u_j\in U$ and $|N[u]|=\chi(G)$, $\forall\,u\in U$.

It is obvious that a vertex which yields a rainbow neighbourhood cannot be an element in the set of faded vertices. 

In this context, we recall an important theorem from \cite{KNB}.

\begin{theorem}{\rm \cite{KNB}}
For cycle $C_n$ with $n$ is odd and $\ell = 0,1,2,\ldots$,  we have
\begin{enumerate}\itemsep0mm 
\item[(i)]~ $r^+_\chi(C_{7+4\ell}) = 3 +2(\ell + 1)$; and
\item[(ii)]~ $r^+_\chi(C_{9+4\ell}) = 3 +2(\ell + 1)$.
\end{enumerate}
\end{theorem}

\begin{proposition}\label{Prop-2.2}
\begin{enumerate}\itemsep0mm 
\item[(i)]~ Let $\mu(G)$ be the Mycielski graph of $G$ of order $n \geq 1$, then $f^-(\mu(G))=f^+(\mu(G)) = n$.
\item[(ii)]~ For bipartite graphs, $f^-(G)=f^+(G)=0$.
\item[(iii)]~ For odd cycles, we have
\begin{enumerate}\itemsep0mm 
\item[(a)]~ $C_3$ and $C_5$, we have $f^-(C_3)=f^-(C_5)=0$ and for $n\geq 7$, we have $f^-(C_n)=n-5$.
\item[(b)]~ For all odd cycles $C_n$, we have $f^+(C_n)=0$.
\end{enumerate}
\item[(iv)]~ Thorn graphs $G^\star$ of $\chi$-colourable graphs $G$ with $\chi(G) \geq 3$, we have
\begin{enumerate}\itemsep0mm 
\item[(a)]~ $f^-(G^\star) \leq f^-(G) + \sum\limits_{i=1}^{\ell}t_i$ and,
\item[(b)]~ $f^+(G^\star) \leq f^+(G) + \sum\limits_{i=1}^{\ell}t_i$, where $t_i \geq 1$ is the number of thorns added to some vertices $v_i \in V(G)$, $1\leq i \leq \ell \leq n$.
\end{enumerate}
\end{enumerate}
\end{proposition} 
\begin{proof}~
\begin{enumerate}\itemsep0mm 
\item[(i)]~ The Mycielski graph or or Mycielskian of a graph $G$, denoted by $\mu(G)$, is the simple connected graph with $V(\mu(G))=V(G) \cup \{x_1,x_2, x_3, ..., x_n\} \cup \{w\}$ and $E(\mu(G))=E(G)\cup \{v_ix_j, v_jx_i|$ if and only if $v_iv_j \in E(G)\} \cup \{wx_i| \forall i\}$ with the known meaning of vertices $x_i$, $1\leq i\leq n$ and $w$. We know that $\chi(\mu(G))=\chi(G)+1$ with $c(w) = c_{\chi(G) +1}$. Therefore, only vertex $w$ yields a rainbow neighbourhood in $G$. Since $w$ is adjacent to all $x_i$, $1\leq i \leq n$ and is not adjacent to any $v\in V(G)$, all solid colours assigned to $v\in V(G)$ may fade to $c^\circ$. Hence the result.

\item[(ii)]~ For bipartite graphs $G$ each vertex $v \in V(G)$ yields a rainbow neighbourhood (see \cite{KNJ} for the proof). Hence, we have $f^-(G)=f^+(G)=0$.

\item[(iii)]~ Consider an odd cycle $C_n$. Then, we have 

\begin{enumerate}\itemsep0mm
\item[(a)]~ Since, $C_3$ is a complete graph the result is trivial. For $C_5$, consider the cyclically colouring $c(v_1) = c_1$, $c(v_2) = c_2$, $c(v_3) = c_1$, $c(v_4) = c_2$ and $c(v_5) = c_3$. Clearly, vertices $v_1,v_4,v_5$ yield rainbow neighbourhoods to result in, $r^-\chi(C_5) = 3$. The $5$-path, $v_{n-2}v_{n-1}v_nv_1v_2$ has to retain solid colours to ensure no decrease in $r^-_\chi(P_5)$ if fading effects at some vertices. Therefore the result for $C_5$. Through immediate induction the result holds for all $n \geq 7$ and $n$ is odd.

\item[(b)]~ Consider the conventional vertex labeling of a cycle $C_n$ to be consecutively and clockwise, $v_1,v_2,v_3,\ldots ,v_n$. For the cycle $C_3$, the result is obvious. For cycle $C_5$, the colouring $c(v_1) = c_1$, $c(v_2) = c_2$, $c(v_3) =c_3$, $c(v_4) = c_1$, $c(v_5)  =c_2$ gives the result.

\noindent For $n\geq 7$ consider three classes of odd cycles.

\begin{enumerate}\itemsep0mm 
\item[(i)]~ For For $C_n$, $n = 3t$, $t\geq 3$ The colouring $c(v_1) = c_1$, $c(v_2) = c_2$, $c(v_3) =c_3, \cdots,~c(v_{n-2}) = c_1$, $c(v_{n-1}) = c_2$, $c(v_n)=c_3$ gives the result.

\item[(ii)]~ For $C_n$, $n = 7 + 4\ell$, $n \neq 3t$, $t\geq 3$, $\ell = 0,1,2,\ldots$ consider the colouring, $c(v_1) = c_1$, $c(v_2) = c_2$, $c(v_3) =c_3,\cdots,~ c(v_{n-1}) = c_3$, $c(v_n) = c_2$ and the result follows through immediate induction.
\item[(iii)]~ For $C_n$, $n = 9 + 4\ell$, $n \neq 3t$, $t\geq 3$, $\ell = 0,1,2,\ldots$ consider the colouring, $c(v_1) = c_1$, $c(v_2) = c_2$, $c(v_3) =c_3,\cdots,~ c(v_{n-2}) = c_3$, $c(v_{n-1}) = c_1$, $c(v_n) = c_2$ and the result follows through immediate induction.
\item[(iv)]~ Consider a thorn graph $G$ such that every internal vertex $v_i$ has $t_i\ge 1$ pendant edges (thorns) attached to it. Then, we have
\end{enumerate}
\begin{enumerate}\itemsep0mm 
\item[(a)]~ Clearly and without loss of generality, all thorns $t_i$, $2\leq i\leq \ell$ may be coloured $c_1$ and thorns $t_1$ may be coloured say, $c_2$. Since $\chi(G)\geq 3$, no thorn vertex can yield a rainbow neighbourhood. Also, a vertex in the corresponding fade set $F^\circ(G)$ cannot yield a rainbow neighbourhood. Thus, since $t_i \geq 1$ it implies that $f^-(G^\star) \geq f^-(G) +1$. However additional vertices can yield rainbow neighbourhoods in $G^\star$ whilst not yielding in $G$. Therefore, $f^-(G^\star) \leq f^-(G) + \sum\limits_{i=1}^{\ell}t_i$.\\
\item[(b)]~ Similar reasoning as that in (vi)(a) applies.
\end{enumerate}
\end{enumerate}
\end{enumerate}

This completes the proof.
\end{proof}

\noi Note that $0 \leq f^+(G)\leq f^-(G) \leq n-1$.

\begin{lemma}\label{Lem-2.3}
For any graph $G$, we have $f^-(G+K_1)= f^-(G)$ and $f^+(G+K_1)= f^+(G)$.
\end{lemma}
\begin{proof}~
Let the graph $G$ permit the chromatic colouring on the colours $\mathcal{C}=\{c_1,c_2,c_3,\ldots,c_{\chi(G)}\}$. The result is obvious for a trivial graph. Hence, assume that $G$ is a non-trivial graph. Therefore, since $G$ is connected, we have $\chi(G) \geq 2$. Certainly, the graph $G'=K_1+G$ requires the chromatic colouring on the colour set $\{c_1,c_2,c_3,\ldots,c_{\chi(G)},c_{\chi(G)+1}\}$. Without loss of generality, let $c(K_1) = c_{\chi(G)+1}$. Obviously, any vertex $v\in V(G)$ that may fade to  a transparent vertex in $G$, may fade in $G'$ also. Therefore, $f^-(G') \geq f^-(G)$ and $f^+(G') \geq f^+(G)$. Since the vertex corresponding to $K_1$ also yields a rainbow neighbourhood, it cannot fade to a transparent vertex. Hence, $f^-(G') \leq f^-(G)$  and $f^+(G') \leq f^+(G)$ This implies $f^-(G')=f^-(G)$ and $f^+(G')=f^+(G)$. Hence the result.
\end{proof}

Lemma \ref{Lem-2.3} implies that if multiple, say $t$, copies of $K_1$ are joined to $G$ then, $f^-(t\cdot K_1+G)=f^-(G)$ and $f^+(t\cdot K_1 + G)=f^+(G)$. 

\noi The \textit{generalised  windmill}-on-$G$, denoted by $W^{(m)}_G$, is the graph defined by $W^{(m)}_G = K_1 + \bigcup\limits_{i=1}^{m}G_i$ with $G_i\simeq G$, $1\leq i\leq m$. Lemma \ref{Lem-2.3} implies that $f^-(W^{(m)}_G) = m\cdot f^-(G)$ and $f^+(W^{(m)}_G) = m \cdot f^+(G)$. It is also to be noted that all known windmill graphs are windmill-on-$G$ for some $G$. The Dutch windmill graph is simply $m$ copies of a cycle $C_n$, $n\geq 3$ which share a common cycle vertex.

\noi Let us now recall the definitions of the join and corona of two graphs. 

The \textit{join} of two graphs $G_1$ and $G_2$ with disjoint vertex sets $V_1$ and $V_2$, and edge sets $E_1$ and $E_2$ is the graph union $G_1\cup G_2$ together with all the edges joining $V_1$ and $V_2$ (see \cite{FH}). The join of $G_1$ and $G_2$ is denoted by $G_1+G_2$. 

The {\em corona} of two graphs $G_1$ and $G_2$, denoted by $G_1\circ G_2$, is the graph obtained by taking one copy of $G_1$ (which has $n_1$ vertices) and $n_1$ copies of $G_2$ and then joining the $i$-th vertex of $G_1$ to every point in the $i$-th copy of $G_2$ (see \cite{FH}).

\noi The next results is on the fading numbers of the join and corona of two graphs.

\begin{theorem}\label{Thm-2.4}
Consider two graphs $G$ and $H$ of order $n_1,n_2$, respectively. Then, we have
\begin{enumerate}\itemsep0mm
\item[(i)]~ $f^-(G+H) = f^-(G)+f^-(H)$ and $f^+(G+H)=f^+(G)+f^+(H)$.
\item[(ii)]~ 
\begin{enumerate}\itemsep0mm
\item[(a)]~ $f^-(G\circ H) = n_1\cdot f^-(H))$ and $f^+(G\circ H) = n_1\cdot f^+(H))$, if $\chi(H) \geq \chi(G)-1$; else
\item[(b)]~ $f^-(G\circ H) \leq f^-(G) + n_1\cdot n_2$ and $f^+(G\circ H) \leq f^+(G)+ n_1\cdot n_2$. 
\end{enumerate}
\end{enumerate}
\end{theorem}
\begin{proof}~ 
Part (i): Without loss of generality, let $G$ permit a chromatic colouring on the colours $\mathcal{C}_1=\{c_1,c_2,c_3,\ldots,c_{\chi(G)}\}$ and let $H$ permit a chromatic colouring on the colours $\mathcal{C}_2=\{c_{\chi(G)+1},c_{\chi(G)+2},c_{\chi(G)+3}, \ldots, c_{\chi(G)+\chi(H)}\}$. Therefore, the graph $G+H$ will permit the chromatic colouring on the colours $\mathcal{C}=\{c_1,c_2,c_3,\ldots,c_{\chi(G)},c_{\chi(G)+1},\\ c_{\chi(G)+2},c_{\chi(G)+3}, \ldots, c_{\chi(G)+\chi(H)}\}$. Clearly, any vertex $v \in V(G)$, that yields a rainbow neighbourhood in $G$, also yields a rainbow neighbourhood in $G+H$ and vice versa. Also, any vertex $u \in V(G)$ that does not yield a rainbow neighbourhood in $G$, cannot yield a rainbow neighbourhood in $G+H$ and vice versa. Similarly, any vertex $v \in V(G)$ that may fade to a transparent vertex in $G$  may fade to transparent in $G+H$ and vice versa. Therefore, the results for both, $f^-(G+H)$, $f^+(G+H)$.

Part (ii)(a): For any vertex $v \in V(G)$ with $c(v) = c_i$ recolour all vertices $u \in V(H)$ which have $c(u) = c_i$ to the colour $c_{\chi(H) + 1}$. Then the first part of the result i.e. for $\chi(H) \geq \chi(G)-1$, is a direct consequence of Lemma \ref{Lem-2.3}. 

Part (ii)(b): For the second part it is clear that all vertices $v \in V(G)$ that yield a rainbow neighbourhood will yield a rainbow neighbourhood in $G\circ H$. Therefore, $r_\chi(G\circ H) \geq r_\chi(G)$. 

It also follows that no vertex $w \in V(H)$ can yield a rainbow neighbourhood in $G\circ H$. To show the aforesaid, assume that the vertex $w \in V(H)$ of the $t$-th copy of $H$ joined to $v \in V(G)$ is a vertex yielding a rainbow neighbourhood in $G\circ H$.  It means that vertex $w$ has at least one neighbour for each colour $c_i,\ 1 \leq i \leq \chi(H) < \chi(G)-1$ as well as the neighbour $v$ with, without loss of generality, the colour $c(v)=c_{\chi(H)+1}$. Since, $c_{\chi(H)+1}$ can at best be the colour $c_{\chi(G)-1}$, the colour $c_{\chi(G)} \notin N[w]$ in $r_\chi(G\circ H)$, which is a contradiction. Therefore, $r_\chi(G\circ H) = r_\chi(G)$.

Hence, $f^-(G\circ H) \geq f^-(G)$ and $f^+(G\circ H) \geq f^+(G)$. Clearly, for any copy of graph $H$, each vertex $w \in V(H)$ may fade to a transparent vertex because any colour $c(w)\in c(G)$. A vertex in $V(G)$, that did not yield a rainbow neighbourhood in $G$, may yield a rainbow neighbourhood in $G\circ H$ under the chromatic colouring of the copies of $H$ and hence it follows that, $f^-(G\circ H) \leq f^-(G) + n_1\cdot n_2$ and $f^+(G\circ H) \leq f^+(G)+ n_1\cdot n_2$. 
\end{proof}

\noi Next, we have our important result as follows:

\begin{theorem}\label{Thm-2.5}
For a graph $G$ of order $n$ we have
\item[(i)]~ $f^-(G)>0$ if and only if $r_\chi^-(G)<n$;
\item[(ii)]~ $f^+(G)>0$ if and only if $r_\chi^+(G)<n$.
\end{theorem}
\begin{proof}~
\textit{Part (i):} Consider a graph $G$ of order $n$ with $r_\chi^-(G)< n$. Consider any vertex $v$ that does not yield a rainbow neighbourhood in $G$. If $v \notin N[u]$ for any $u$ which yields a rainbow neighbourhood, then $v$ may fade to become a transparent vertex. Hence, $f^-(G)>0$. Also, if each $u$ that yields a rainbow neighbourhood and has a vertex $v \in N[u]$ and a vertex $w \neq v$ such that $c(w) = c(v)$, then $v$ may fade to become transparent. Therefore, $f^-(G)>0$. Finally, in each $N[u]$, if $v \in N[u]$, $u$ yields a rainbow neighbourhood and the colour $c(v)$ is distinct from all colours in $\mathcal{C}$, then the colouring is not minimum as assumed. Hence, $v$ can be recoloured such that $\chi(G)$ remains valid and the colouring becomes minimum. Therefore, $v$ is  permitted to fade to become transparent vertex. Alternatively, $v$ is adjacent to at least one of each colours in $\mathcal{C}$. This implies that $v$ indeed yields a rainbow neighbourhood in $G$ and by immediate induction all vertices yield a rainbow neighbourhood in $G$. Hence, $r^-_\chi(G)=n$, which is a contradiction. Therefore, if $r_\chi^-(G)<n$ then, $f^-(G)>0$.

For the converse, consider a graph $G$ of order $n$ for which $f^-(G)>0$. The result is straight forward, because any vertex that yields a rainbow neighbourhood cannot fade to transparent.

\textit{Part (ii):} This part follows by similar reasoning of that in Part (i).
\end{proof}

\section{Conclusion}

In this paper, we have introduced the notion of the fading number of a graph $G$ and discussed this parameter for certain fundamental graph classes. We note that there is a wide scope for further research to determine the minimum and maximum fading numbers for different classes of graphs. One important open problem we have identified during our study is to resolve, if possible, is to find the respective minimum and maximum rainbow neighbourhood numbers from the adjacency matrix of a graph $G$. 

Recall that the clique number of a graph $G$ is the order of a maximal clique (induced complete graph) in $G$. The clique number is denoted by $\omega(G)$.

\begin{conjecture}{\rm 
For a graph $G\ne C_n$; $n$ is odd, such that $\omega(G)< \chi(G)$, we have $f^+(G)>0$.
}\end{conjecture}

We have discussed some results on the join and corona of two graphs in this paper. The study may be extended to other graph operations and graph products. A study on the fading number of different graph powers are also possible. All these facts highlight a wide scope for further studies in this area.

\section*{Acknowledgement}

Authors of the paper would like to gratefully acknowledge the critical and creative comments of the anonymous referee(s) which helped to improve the content and presentation of the paper significantly.

\end{document}